\documentclass[12pt, a4paper]{article}
\usepackage{pdfsync}
\usepackage{amsmath}\multlinegap=0pt
\usepackage[latin1]{inputenc}
\usepackage{amsthm}
\usepackage{amssymb}
\usepackage[francais,english]{babel}
\usepackage{color}

\numberwithin{equation}{section}
\theoremstyle{plain}
\newtheorem{thm}{Theorem}[section]

\newtheorem{lem}[thm]{Lemma}

\theoremstyle{definition}

\theoremstyle{remark}
\newtheorem{rem}[thm]{Remark}

\newcommand{\R}{\mathbb{R}}
\newcommand{\T}{\mathbb{T}}

\newcommand\N{{\mathbb N}}
\newcommand\Z{{\mathbb Z}}

\newcommand\pref[1]{(\ref{#1})}
\newcommand\loc{{\mathrm{loc}}}

\let \wto\rightharpoonup
\let \eps\varepsilon

\DeclareMathOperator{\ent}{Ent}

\newcommand\dive{\mathrm{div}}

\def\<#1,#2>{\left<#1,#2\right>}

\let \wto\rightharpoonup

\def\PP{{\mathcal P}}
\def\E{{\mathcal E}}
\def\U{{\mathcal U}}

\def\PT{\PP(\T^d)}
\def\PTac{\PP_{\mathrm{ac}}(\T^d)}

\title{On systems of continuity equations with nonlinear diffusion and nonlocal drifts}

\author {Guillaume Carlier \thanks{\scriptsize CEREMADE, UMR CNRS 7534, Universit\'e Paris IX
Dauphine, Pl. de Lattre de Tassigny, 75775 Paris Cedex 16, FRANCE
\texttt{carlier@ceremade.dauphine.fr}},
Maxime Laborde \thanks{\scriptsize CEREMADE, UMR CNRS 7534, Universit\'e Paris IX
Dauphine, Pl. de Lattre de Tassigny, 75775 Paris Cedex 16, FRANCE
\texttt{laborde@ceremade.dauphine.fr}.}}

\begin{document}

\maketitle

\begin{abstract}
This paper is devoted to existence and uniqueness results for classes of nonlinear diffusion equations (or systems) which may be viewed as \emph{regular} perturbations of Wasserstein gradient flows. First, in the  case where the drift is a gradient (in the physical space), we obtain existence by a semi-implicit Jordan-Kinderlehrer-Otto scheme. Then, in the nonpotential case, we derive existence from a regularization procedure and parabolic energy estimates.  We also address the uniqueness issue by a displacement convexity argument. 

\end{abstract}


\medskip

\textbf{MS Classification:} 35K15, 35K40, 49J40.

\section{Introduction}

The continuity equation with a density-dependent drift
\[\partial_t \rho=\dive(\rho v), \; \mbox{ with }  v=V[\rho]\]
is ubiquitous in  modeling and arises in a variety of domains such as biology, particle physics, population dynamics, crowd modelling, opinion formation... It should actually come as no surprise since it captures the dynamics of a  population of particles following the ODE $\dot{X}=-v(t,X)$ where $v=V[\rho]$ depends itself on the density in a way (local, nonlocal, attractive, repulsive etc..) depending on which phenomena (aggregation, diffusion...) one aims to capture and the type of applications. 

\smallskip

Of course, at this level of generality not much can be said on existence and uniqueness. There are however two cases which may be treated in a systematic way. The first one, is the \emph{regular case} where $V[\rho]$ is  a smooth vector field whatever the probability measure $\rho$ is, with some uniform bounds on some of its derivatives  and $\rho \mapsto V[\rho]$ is Lipschitz in the Wasserstein metric. In this regular case, existence and uniqueness can be proved as a simple exercise by the method of characteristics and a suitable fixed point argument. This regular case (a typical example being that of a convolution) is however rather restrictive and for instance rules out diffusion. The second case where there is a general theory is the\emph{Wasserstein gradient flow case}. In this case, at least at a formal level, $v$ may be written as $v=\nabla \frac{\delta \E} {\delta \rho}$ that is the gradient of the first variation of a functional $\E$ defined on measures. In their seminal paper \cite{jko}, Jordan, Kinderlehrer and Otto  discovered that the heat flow is the gradient flow of the entropy functional  $\E(\rho)=\int \rho \log(\rho)$ which corresponds to the case $v=\frac{\nabla \rho}{\rho}$. The theory of Wasserstein gradient flows has been very succesful in addressing a variety of nonlinear evolution equations such as the porous medium equation \cite{ottopm}, aggregation equations \cite{cdfls} or granular media equations \cite{cmv}. This powerful theory is presented in a complete and detailed way in the reference book of Ambrosio, Gigli and Savar\'e \cite{ags}.

\smallskip

The purpose of the present paper is  a contribution to the following general question: can one hope for an existence/uniqueness theory in the case where $V$ is the sum of a Wasserstein gradient flow term and a regular term (not necessarily a gradient).  Our motivation for this question actually comes from systems. For instance,  a  simple but natural model, for the evolution of two (say) interacting species is:
\[\begin{cases}
\partial_t \rho_1= \nu_1 \Delta \rho_1 + \dive(\rho_1 \nabla( F \star \rho_1+ G\star \rho_2)), \\
\partial_t \rho_2=\nu_2 \Delta \rho_2 +\dive(\rho_2 \nabla (H\star \rho_1+K\star \rho_2)).
\end{cases}\]
 When $\nu_1=\nu_2=0$ i.e. without diffusion, this is exactly the system studied by Di Francesco and Fagioli \cite{dff}. As emphasized in \cite{dff}, if  cross-interactions are symmetric i.e.  $G=K$ (or more generally $G$ and $K$ are proportional), this system has a (product)  Wasserstein gradient flow structure but this is certainly a restrictive and often unrealistic assumption in applications. This is why Di Francesco and Fagioli, still taking advantage of the similarity with Wasserstein gradient flows used a semi-implicit scheme  \`a la Jordan-Kinderlehrer-Otto  to obtain existence and uniqueness results. In \cite{dff}, there is no diffusion but clearly the structure of the system belongs to the \emph{mixed} case where drifts can be decomposed as the sum of a Wasserstein gradient and a regular term. Of course, the semi-implicit scheme only makes sense when drifts are gradients.

 \smallskip

Regarding systems with a gradient structure and in the presence of nonlinear diffusion, our first contribution is to establish strong enough convergence at the level of  the semi-implicit scheme to recover a solution of the PDE at the limit. The delicate step is of course to pass to the limit in the nonlinear diffusion term, which can be done thanks to the powerful \emph{flow interchange} argument of Matthes, McCann and Savar\'e \cite{mms} in a similar way as in the work of Di Francesco and Matthes \cite{dfm}. We will then address the nonpotential case in which the drift may contain a nongradient (but regular) part. This case cannot be attacked by the semi-implicit minimization scheme and we will prove existence  by suitably regularizing the diffusion and using standard parabolic energy estimates. Finally, we will derive an uniqueness result from displacement convexity of the energy.

The paper is organized as follows. In section \ref{sec-jko-si}, we consider the potential case, introduce a semi-implicit scheme \`a la Jordan-Kinderlehrer-Otto \cite{jko} and state a first existence result. Section \ref{sec-proof} is devoted to the proof of this existence result. Section \ref{sec-systems} extends the result to the case of systems (again in the case where all drifts are gradients). Section \ref{sec-nonpot} proves existence for the non-potential case. The final section \ref{sec-uniq} shows uniqueness by a simple displacement convexity argument.

\section{The potential case and the semi-implicit JKO scheme}\label{sec-jko-si}

Our aim is to solve the following nonlinear diffusion equation:
\begin{equation}\label{nonlineardiffpot}
\partial_t \rho=\dive\Big(\rho \nabla (E'(\rho) + U[\rho])\Big), \; \rho\vert_{t=0}=\rho_0,
\end{equation}
on $(0,+\infty)\times \T^d$ where $\T^d:=\R^d/\Z^d$ denotes the flat torus (we take periodic boundary conditions to simplify the exposition, we refer to the work of the second author \cite{laborde} for extensions to $\R^d$ or a bounded domain) which we identify with the unit cube $[0,1]^d$ equipped with the quotient distance:
\[d^2(x,y):=\inf_{k\in \Z^d}  \vert x-y+k\vert^2.\] 
Denoting by $\PT$ the set of Borel probability measures on $\T^d$, we assume the following assumption on the map $\rho \in \PT\mapsto U(\rho)$: 
\begin{equation}\label{hypv0pot}
\forall \rho \in \PT,   U[\rho] \in W^{2, \infty}(\T^d),  \mbox{ and }  U[\rho]\ge 0, 
\end{equation}
\begin{equation}\label{hypv1pot}
  \sup_{\rho \in \PT} \{  \Vert \nabla U[\rho] \Vert_{L^\infty} +  \Vert (\Delta  U[\rho])_+ \Vert_{L^\infty} \}<+\infty
\end{equation}
and there exists a constant $C$ such that for all $(\rho, \nu)\in \PT\times \PT$
\begin{equation}\label{hypv2pot}
\Vert \nabla U[\rho]-\nabla U [\nu]\Vert_{L^{\infty}(\T^d)} \le C W_2(\rho,\nu),
\end{equation}
with $W_2(\rho,\nu)$ denoting the $2$-Wasserstein distance between $\rho$ and $\nu$ i.e. 
\[W_2(\rho,\nu):= \inf_{\gamma\in \Pi(\rho, \nu) }  \Big\{\int_{\T^d\times \T^d} d^2(x,y) \mbox{d} \gamma(x,y) \Big\}^{\frac{1}{2}} \]
where $ \Pi(\mu, \nu)$ is the set of transport plans between $\rho$ and $\nu$ i.e. the set of  Borel probability measures on $\T^d\times \T^d$ having $\mu$ and $\nu$ as marginals. It is well known, see \cite{villani, villani2, fs}, that $W_2$ metrizes the weak star topology on $\PT$ so that $(\PT, W_2)$ is a compact metric space.

As for the nonlinear diffusion term $\dive(\rho  \nabla (E'(\rho)))$ it is convenient to rewrite it as:
\[\dive(\rho  \nabla (E'(\rho)))=\Delta F'(\rho)\]
where $F'(t):=tE'(t)-E(t)$ so that $F''(t)=tE''(t)$. The typical energies $E$ we have in mind are the following classical examples
\begin{itemize}
\item Entropy: $E(t):=t\log(t)$ so that $F'(t)=t$, $F''(t)=1$ (which thus gives a linear diffusion driven by the laplacian),

\item Porous media $E(t)=t^m$ with $m>1$ so that $F'(t)=(m-1)t^m$, $F(t)=\frac{m-1}{m+1} t^{m+1}$.

\end{itemize}

We shall assume that $E$ is a continuous convex function on $\R_+$ with $E(0)=0$, $E$ is of class $C^2$ on $(0,+\infty)$ and that there are constants  $C>0$ and $m\ge 1$ such that
\begin{equation}\label{hypEpot}
E''(t)\ge \frac{t^{m-2}}{C}, \; F'(t)=tE'(t)-E(t)\le C(1+t^m), \; \forall t\in (0,+\infty).
\end{equation}
Of course, these assumptions are satisfied in the examples above corresponding respectively to linear diffusion and the porous medium equation. 

Finally, we assume that the initial condition $\rho_0\in \PT$ satisfies
\begin{equation}\label{hyprho0pot}
\int_{\T^d} E(\rho_0(x)) \mbox{ d} x<+\infty
\end{equation}
which with \pref{hypEpot} in particular implies that $\rho_0\in L^m(\T^d)$ and $F'(\rho_0)\in L^1(\T^d)$.  

A weak solution of \pref{nonlineardiffpot} then is a  curve $t\in (0,+\infty)\mapsto \rho(t,.)\in \PT$ such that $F'(\rho) \in L^1_{\loc}((0,+\infty)\times \T^d)$ and 
\begin{equation}\label{solfaiblepot}
\int_0^{+\infty} \int_{\T^d}  (\partial_t \phi  \rho +\Delta \phi F'(\rho)-\nabla U[\rho] \cdot \nabla \phi \rho) \mbox{ d} x  \mbox{ d} t=-\int_{\T^d} \phi(0,x) \rho_0(x)  \mbox{ d} x
\end{equation}
for every $\phi\in C_c^2([0,+\infty)\times \T^d)$.

\begin{thm}\label{existporouspot}
Assume \pref{hypv0pot}-\pref{hypv1pot}-\pref{hypv2pot}-\pref{hypEpot}-\pref{hyprho0pot}, then \pref{nonlineardiffpot} admits at least one weak solution. 
\end{thm}

The complete proof of this result is given in section \ref{sec-proof}. This proof is strongly based on a semi-implicit version of the Jordan-Kinderlehrer-Otto (JKO) scheme \cite{jko} as in Di Francesco and Fagioli \cite{dff}. More precisely given a time step $h>0$, we construct inductively a sequence $\rho_h^k\in \PT$ by setting $\rho_h^0=\rho_0$ and, given $\rho_h^k$ we select $\rho_h^{k+1}$ as a solution of 
\begin{equation}\label{semi-imp-jko}
\inf_{\rho\in \PT} \Big\{\frac{1}{2h} W_2^2 (\rho, \rho_h^k)+ \E(\rho)+ \U(\rho\vert \rho_{h}^k)\Big\}
\end{equation}
where
\begin{equation}\label{defdesfonctionnelles}
\E(\rho):= \begin{cases} \int_{\T^d} E(\rho(x)) \mbox{ d} x, \; \mbox{ if $E(\rho)\in L^1$},\\+\infty, \mbox{ otherwise,} \end{cases}   \U(\rho\vert \nu):=\int_{\T^d} U[\nu] \mbox{d} \rho.
\end{equation}
Note that assumption \pref{hypEpot} ensures that $\E$ controls from above  $\int_{\T^d} \rho^m \mbox{d} x$ if $m>1$ and $\int_{\T^d} \rho \log(\rho) \mbox{d} x$ if $m=1$, so in any case sublevels of $\E$ are weakly relatively compact in $L^m(\T^d)$. 

By standard lower semicontinuity and compactness arguments, it is clear that \pref{semi-imp-jko} possesses solutions so that one can indeed generate a sequence $(\rho_h^k)_{k\in \N}$ by the semi-implicit JKO scheme \pref{semi-imp-jko}. It is even uniquely defined (but we won't really need it in the sequel) because each $\rho_h^k$ remains absolutely continuous with respect to Lebesgue's measure, so that $\rho\mapsto W_2^2 (\rho, \rho_h^k)$ is in fact strictly convex (see Proposition 7.19 of \cite{fs}) and the other terms $\E$ and $\U(.\vert \rho_h^k)$ are convex. We finally extend in a piecewise constant way the sequence $(\rho_h^k)_{k\in \N}$ i.e. set:
\begin{equation}
\rho_h(t,.):=\rho_h^k \mbox{ for $t\in ((k-1)h, kh]$ and $k\in \N$}.
\end{equation}
The proof detailed in the next section consists in showing that as $h\to 0$, one may recover a limit $\rho$ which satisfies \pref{nonlineardiffpot}. This is the same strategy as in \cite{dff} but the tricky part consists in passing to the limit in the nonlinear diffusion term $F'(\rho_h)$. This will be done thanks to the powerful flow interchange argument of Matthes, McCann and Savar\'e \cite{mms} in a similar way as in the work of Di Francesco and Matthes \cite{dfm}.

\section{Existence proof}\label{sec-proof}

The proof of Theorem \ref{existporouspot} is divided into three parts. The first two parts concern a priori estimates on $\rho_h$ and the last part consists in showing that passing to the limit in the Euler-Lagrange equation of \pref{semi-imp-jko} actually enables us to recover a solution  of  \pref{nonlineardiffpot}. The discussion on uniqueness is deferred to the final section \ref{sec-uniq}. Of course, it is enough to work on a fixed finite time interval $(0,T)$, which we shall implicitly do below, we thus also set $N:=[\frac{T}{h}]+1$. In what follows $C$ (respectively $C_T$) is a generic  (resp. only depending on $T$) constant whose value may vary from one line to another

\subsection{Basic a priori estimates}

From the very definition of the JKO semi-implicit scheme \pref{semi-imp-jko} we have for every $k$:
\begin{equation}\label{estimjko1}
\frac{1}{2h} W_2^2 (\rho_h^{k+1}, \rho_h^k)\le \E(\rho_h^k)-\E(\rho_h^{k+1}) +\int_{\T^d} U[\rho_h^k] \mbox{d} (\rho_h^k-\rho_h^{k+1}).
\end{equation}
Recall then that the $1$-Wasserstein distance $W_1$ is defined by:
\[W_1(\rho,\nu):=\inf_{\gamma\in \Pi(\rho, \nu) }    \Big\{\int_{\T^d\times \T^d}  d(x,y) \mbox{d} \gamma(x,y) \Big\}, \]
so that by Cauchy-Schwarz's inequality
\[ W_1(\rho,\nu) \le W_2(\rho,\nu).\]
The well-known Kantorovich duality (see \cite{villani, villani2, fs}) states that $W_1$ can be also be expressed as
\[W_1(\rho,\nu)=\sup \Big\{\int_{\T^d} \phi \mbox{d} (\rho-\nu), \; \phi \mbox{ $1$-Lipschitz on $\T^d$}\Big\}\]
so that $\int_{\T^d} \phi \mbox{d} (\rho-\nu)$ is less than the Lipschitz constant of $\phi$ times $W_1(\rho,\nu)$. Thanks to these considerations, assumption \pref{hypv1pot} 
and Young's inequality, we get
\begin{equation}\label{estimjko2}
\int_{\T^d} U[\rho_h^k] \mbox{d} (\rho_h^k-\rho_h^{k+1}) \le C W_2 (\rho_h^{k+1}, \rho_h^k)\le \frac{1}{4h} W_2^2 (\rho_h^{k+1}, \rho_h^k)+ C^2h.
\end{equation}
Together with \pref{estimjko1}, this gives
\begin{equation}\label{estimjko3}
\frac{1}{4h} W_2^2 (\rho_h^{k+1}, \rho_h^k)\le \E(\rho_h^k)-\E(\rho_h^{k+1})+Ch
\end{equation}
summing between $0$ and $N$ and using the fact that $E$ is bounded from below and \pref{hyprho0pot} gives 
\begin{equation}\label{estimjko4}
\frac{1}{4h} \sum_{k=0}^{N-1} W_2^2 (\rho_h^{k+1}, \rho_h^k)\le \E(\rho_0)-\E(\rho_h^{N})+CNh\le C(1+T),
\end{equation}
as well as
\begin{equation}\label{estimjko5}
\E(\rho_h^k)\leq \E(\rho_0)+Ckh,
\end{equation}
which, thanks to \pref{hyprho0pot} and \pref{hypEpot}, also gives
\begin{equation}\label{estimjko6}
\sup_{t\in[0,T]} \Vert \rho_h(t,.)\Vert^m_{L^m} \le C(1+T) \mbox{ if $m>1$},
\end{equation}
and
\begin{equation}\label{estimjko6m=1}
\sup_{t\in[0,T]} \int_{\T^d} \rho_h(t,x)\log(\rho_h(t,x)) \mbox{d} x \le C(1+T), \mbox{ if $m=1$.} 
\end{equation}

With \pref{estimjko4}, we also have the H\"{o}lder like estimate:
\begin{equation}\label{estimjko7}
W_2(\rho_h(t,.), \rho_h(s,.))\le C\sqrt{(1+T)} \sqrt{\vert t-s\vert+h}, \; \forall (s, t) \in [0,T]^2. 
\end{equation}
Using \pref{estimjko7} and refined versions of Ascoli-Arzel\`a Theorem (see \cite{ags}) and \pref{estimjko6}-\pref{estimjko6m=1}, one  deduces the existence of a vanishing sequence $h_n\to 0$ and of a $\rho\in C^{0,1/2}([0,T], (\PT, W_2))\cap L^{\infty}((0,T), L^m(\T^d))$ such that 
\begin{equation}\label{estimjko8}
\rho_{h_n} \wto \rho \mbox{ in $ L^m((0,T)\times \T^d)$},  \lim_n \sup_{t\in [0,T]} W_2(\rho_{h_n}(t,.), \rho(t,.)) = 0.
\end{equation}
Now, using \pref{hypv2pot}, we deduce that $\nabla U[\rho_{h_n}]$ converges to  $\nabla U[\rho]$ in  $L^{\infty}((0,T)\times \T^d)$ and then
\begin{equation}\label{estimjko9}
\rho_{h_n}   \nabla U[\rho_{h_n}]  \wto \rho \nabla U[\rho] \mbox{ in $ L^m((0,T)\times \T^d)$}.
\end{equation}

\subsection{Refined a priori estimates by flow interchange}\label{flowint}

This is the key step in the proof which enables us to obtain strong convergence, in what follows we essentially follow similar arguments as in Di Francesco and Matthes \cite{dfm}. For $\nu\in \PT$, set
\[\ent(\nu):= \begin{cases}  \int_{\T^d} \nu(x) \log(\nu(x)) \mbox{d} x, \mbox{ if $\nu \log(\nu)\in L^1$},\\+\infty, \mbox{ otherwise}. \end{cases}\]
For $\nu\in \PT$ with $\ent(\nu)<+\infty$ let us denote by $e^{t\Delta} \nu:=\eta(t,.)$, the solution at time $t$ of the heat equation:
\begin{equation}
\partial_t \eta=\Delta \eta, \; \eta\mid_{t=0}=\nu.
\end{equation} 
It is well-known since the seminal work of \cite{jko} that the heat flow can be viewed as the gradient flow of $\ent$ for $W_2$ (see \cite{ags} for the theory of gradient flows in metric spaces). Moreover the fact that $\ent$ is \emph{displacement convex}\footnote{See section \ref{sec-uniq} for a precise definition. Here, we are working on $\T^d$, but we should not worry about it, it is just if we were working on $\R^d$ with periodic functions only. The optimal transport map between absolutely continuous periodic measures is well-known, it is given by the gradient of a convex function $F$ such that $F(x)-\frac{\vert x\vert^2}{2}$ is periodic (see Cordero-Erausquin \cite{cordero}) and which  is characterized by a Monge-Amp\`ere equation. Displacement convexity of the entropy on the $\T^d$ can therefore be proved as in the euclidean case. Another way to see this is to remark that  Bochner's formula on $\T^d$ is just the same as in $\R^d$, this does not change the Ricci curvature and thus, thanks to a celebrated result of Lott and Villani \cite{lv} and Sturm \cite{sturm} (see in particular the proof of Theorem 4.9), this  does not change the displacement convexity of $\ent$ (with respect to Lebesgue's measure).}, gives (see \cite{ags} Theorem 11.1.4, \cite{sd}, \cite{ow}), the following evolution variational equality:
\begin{equation}\label{evi}
 \frac{1}{2}\frac{d^+}{dt} W_2^2(e^{t\Delta} \nu, \mu) \le \ent(\mu)-\ent(e^{t\Delta} \nu), \; \forall t\ge 0, \; \mu\in \PT
\end{equation}
where we have used the notation:
\[\frac{d^+}{dt} f(t)=\limsup_{s\to 0^+} \frac{f(t+s)-f(t)}{s}.\]
Taking $e^{t\Delta} \rho_h^{k+1}$ as a competitor in the minimization \pref{semi-imp-jko} gives
\begin{equation}
\label{variationpbmin}
\begin{split} 0\le \frac{1}{2h}\frac{d^+}{dt} \Big(W_2^2(e^{t\Delta} \rho_h^{k+1}, \rho_h^k)\Big)_{\vert_{t=0} }+ \frac{d^+}{dt} \Big(\E(e^{t\Delta} \rho_h^{k+1}) \Big)_{\vert_{t=0}} \\
 +  \frac{d^+}{dt} \Big(\U(e^{t\Delta} \rho_h^{k+1}\vert \rho_h^k)\Big)_{\vert_{t=0} }.
 \end{split}
\end{equation}

Since $e^{t\Delta} \rho_h^{k+1}$ is smooth for $t>0$, we can directly compute:
\[\frac{d}{dt} (\E(e^{t\Delta} \rho_h^{k+1}))=-\int_{\T^d} E''(e^{t\Delta} \rho_h^{k+1}) \vert \nabla (e^{t\Delta} \rho_h^{k+1} )\vert^2 \mbox{ dx}\]

which, with \pref{hypEpot} gives that for some positive constant $\lambda>0$
\begin{equation*}
\frac{d}{dt} (\E(e^{t\Delta} \rho_h^{k+1}))  \le  - \lambda \int_{\T^d} \vert \nabla (  (e^{t\Delta} \rho_h^{k+1} )^{\frac{m}{2}} ) \vert^2.
\end{equation*}
We then have 
\begin{equation}
\label{variationenergie}
-\frac{d^+}{dt} \Big(\E(e^{t\Delta} \rho_h^{k+1})\Big) \ge  \lambda \liminf_{s\to 0^+}  \int_0^1  \int_{\T^d} \vert \nabla (  (e^{ts\Delta} \rho_h^{k+1} )^{\frac{m}{2}} ) \vert^2 \mbox{d}x \mbox{d}t.
\end{equation}

In a similar way, for $t>0$,  we have 
\[\frac{d}{dt} \Big(\U(e^{t\Delta} \rho_h^{k+1}\vert \rho_h^k)\Big)=\int_{\T^d} \Delta (U[\rho_h^k]) e^{t\Delta} \rho_h^{k+1}\] 
and the right hand side is uniformly bounded from above thanks to  \pref{hypv1pot}.  With \pref{evi}-\pref{variationpbmin}-\pref{variationenergie} this gives:
\[\lambda h \liminf_{s\to 0^+}  \int_0^1  \int_{\T^d} \vert \nabla (  (e^{ts\Delta} \rho_h^{k+1} )^{\frac{m}{2}} ) \vert^2 \mbox{d}x \mbox{d}t \le (\ent(\rho_h^k)-\ent(\rho_h^{k+1})) +Ch.\]
 Since $e^{s\Delta} \rho_h^{k+1}$ converges strongly to  $\rho_h^{k+1}$ in $L^m$ as $s\to 0^+$, $(e^{s\Delta} \rho_h^{k+1})^{\frac{m}{2}}$ converges strongly to  $(\rho_h^{k+1})^{\frac{m}{2}}$ in $L^2$. By lower semicontinuity we deduce that $\nabla (\rho_h^{k+1} )^{\frac{m}{2}}\in L^2$ and 
\begin{equation}\label{estimjko11}
h \int_{\T^d} \vert \nabla (  \rho_h^{k+1} )^{\frac{m}{2}} ) \vert^2 \mbox{d}x  \le C(\ent(\rho_h^k)-\ent(\rho_h^{k+1}) +h).
\end{equation}
Summing from $k=0$ to $N-1$ and using the fact that $\ent(\rho_0)$ is finite gives 
\[\int_{0}^T \int_{\T^d}  \vert \nabla (  \rho_h )^{\frac{m}{2}} ) \vert^2 \mbox{d}x \mbox{d} t\le CNh +C(\ent(\rho_0)-\ent(\rho_h^N))\le C(1+T)\]
which, with \pref{estimjko6}, also gives 
\begin{equation}\label{sobest}
\int_{0}^T  \Vert \rho_h(t,.)^{\frac{m}{2}} \Vert^2_{H^1(\T^d)}  \mbox{d} t  \le C_T.
\end{equation}
We then observe that since the injection of $H^1(\T^d)$ in $L^2(\T^d)$ is compact and since $\eta\mapsto \eta^{\frac{2}{m}}$ maps continuously $L^2(\T^d)$ into $L^m(\T^d)$, sublevel sets of $\rho \mapsto  \Vert \rho_h(t,.)^{\frac{m}{2}} \Vert_{H^1(\T^d)}$ are strongly relatively compact in $L^m(\T^d)$. Now arguing as in Di Francesco and Matthes \cite{dfm} i.e. using the refined version of Aubin-Lions Lemma provided by Theorem 2  of Rossi and Savar\'e \cite{rs} gives that the family $(\rho_h)_h$ is relatively compact in $L^m((0,T)\times \T^d)$. The conclusion of this step is that \pref{estimjko8} can be strenghtened to 
\begin{equation}\label{estimjko10}
\rho_{h_n} \to \rho \mbox{ strongly in $ L^m((0,T)\times \T^d)$}.
\end{equation}
Now, thanks to the second part of \pref{hypEpot}  and Krasnoselskii's Theorem (see \cite{dfig}, chapter 2) $\rho\mapsto F'(\rho)$ is continuous from $L^m$ to $L^1$ and then \pref{estimjko10} implies that
\begin{equation}\label{estimjko100}
F'(\rho_{h_n}) \to F'(\rho) \mbox{ strongly in $ L^1((0,T)\times \T^d)$}.
\end{equation}

\subsection{Passing to the limit in the Euler-Lagrange equation}

Now, we write the Euler-Lagrange equation for \pref{semi-imp-jko} as in the seminal work of Jordan, Kinderlehrer and Otto \cite{jko} for the Fokker-Planck equation and Otto \cite{otto} for nonlinear diffusions. Let $\xi$ be a $C^\infty$ vector field on $\T^d$ and denote by $X_t$ the flow of $\xi$:
\[\frac{d}{dt} X_t(x)=\xi(X_t(x)), \; X_0(x)=x.\]
Define then $\nu_t:={X_t}_{\#}\rho_h^{k+1}$ so that the change of variables formula gives 
\begin{equation}\label{jacobien}
\rho_h^{k+1}=\nu_t (X_t) \det(DX_t).
\end{equation} 
Since $\rho_h^{k+1}$ solves \pref{semi-imp-jko}, we have
\begin{equation}\label{derivpos}
0\le \frac{1}{2h}\frac{d^+}{dt} \Big(W_2^2(\nu_t, \rho_h^k)\Big)_{\vert_{t=0} } + \frac{d^+}{dt} \Big(\E(\nu_t)\Big)_{\vert_{t=0} }  +\frac{d^+}{dt} \Big(\U(\nu_t \vert\rho_h^k)\Big)_{\vert_{t=0} }.
\end{equation}
Let then $\gamma_h^k\in \Pi(\rho_h^k, \rho_h^{k+1})$ be such that 
\[W_2^2(\rho_h^{k+1}, \rho_h^k)= \int_{\T^d\times \T^d} d^2(x,y) \mbox{d} \gamma_h^k(x,y).\]
  Choosing $k(x,y)\in \Z^d$ such that  $d^2(x,y) =\vert y+k(x,y)- x\vert^2$,  by definition of $\nu_t$, we then have
\[W_2^2(\nu_t, \rho_h^k)\le \int_{\T^d\times \T^d} \vert X_t(y)+k(x,y)- x\vert^2 \mbox{d} \gamma_h^k(x,y)\]
and since $X_t(y)=y+t\xi(y)+o(t)$, we get
\begin{equation}\label{derivestim1}
\frac{1}{2h}\frac{d^+}{dt} \Big(W_2^2(\nu_t, \rho_h^k)\Big)_{\vert_{t=0} } \le \int_{\T^d\times \T^d} \xi(y)\cdot \Big(\frac{y-x}{h}\Big) \mbox{d} \gamma_h^k(x,y). 
\end{equation}
As for the differentiation of $\E(\nu_t)$, following \cite{jko}-\cite{otto}, using \pref{jacobien} we write
\[\E(\nu_t)=\int_{\T^d} E\Big(\frac{\rho_h^{k+1}}{\det(DX_t)}\Big)\det (DX_t)\]
observing that for $\rho\ge 0$ and $\alpha>0$  
\[\frac{d}{d\alpha} ( E\Big(\frac{\rho}{\alpha} \Big)\alpha)=-F'\Big(\frac{\rho}{\alpha} \Big)\]
then thanks to \pref{hypEpot}, \pref{estimjko6}, the fact that $\det(X_t)=1+t\dive(\xi)+o(t)$ (with a uniform  $o(t)$) and Lebesgue's dominated convergence Theorem, one obtains:
\begin{equation}\label{derivestim2}
\frac{d^+}{dt} \Big(\E(\nu_t)\Big)_{\vert_{t=0} }=-\int_{\T^d} F'(\rho_h^{k+1}) \dive(\xi).
\end{equation}
In a similar way, 
\begin{equation}\label{derivestim3}
\frac{d^+}{dt} \Big(\U(\nu_t)  \Big)_{\vert_{t=0} }=\int_{\T^d} \nabla U [\rho_h^k]\cdot \xi \rho_h^{k+1}.
\end{equation}
Combining \pref{derivestim1}-\pref{derivestim2}-\pref{derivestim3} and \pref{derivpos}, and applying the previous to both $\xi$ and $-\xi$ gives the Euler-Lagrange equation
\begin{equation}\label{eul-lag}
 \int_{\T^d\times \T^d} \xi(y)\cdot (x-y) \mbox{d} \gamma_h^k(x,y)=-h\int_{\T^d} F'(\rho_h^{k+1}) \dive(\xi)+h\int_{\T^d} \nabla U [\rho_h^k]\cdot \xi \rho_h^{k+1}.
\end{equation}
for every smooth vector field $\xi$. Now let $\phi\in C_c^{\infty}([0, T)\times \T^d)$ (which we extend by $\phi(0,.)$ on $(-h,0)$), we then have  
\[\begin{split}
\int_0^T \int_{\T^d} \partial_t \phi \rho_h&=\sum_{k=0}^N \int_{\T^d} (\phi(kh,.)-\phi((k-1)h,.)) \rho_h^k\\
&=\sum_{k=0}^{N} \int_{\T^d} \phi(kh,.)(\rho_h^k-\rho_h^{k+1})  -\int_{\T^d} \phi(0,.) \rho_0\\
&=\sum_{k=0}^{N} \int_{\T^d\times \T^d} (\phi(kh,x)- \phi(kh,y))\mbox{d} \gamma_h^k(x,y) -\int_{\T^d} \phi(0,.) \rho_0.
\end{split}\]
Using a second order Taylor-Lagrange formula gives
\[\int_{\T^d\times \T^d} (\phi(kh,x)- \phi(kh,y))\mbox{d} \gamma_h^k(x,y)=\int_{\T^d\times \T^d} \nabla \phi(kh,y)\cdot (x-y)\mbox{d} \gamma_h^k(x,y)+R_h^k\]
with
\[\vert R_h^k \vert \le \Vert D^2 \phi\Vert_{L^\infty} W_2^2(\rho_h^{k+1}, \rho_h^k).\]
With \pref{estimjko3} this gives that $\sum_{k=0}^N \vert R_h^k\vert \le C_T h$, so  that applying  \pref{eul-lag} to $\xi=\nabla \phi(kh,.)$, and using the fact that, with \pref{hypv2pot}, Cauchy-Schwarz inequality and \pref{estimjko4},  $\sum_{k=0}^N \Vert \nabla U[\rho_h^k]-\nabla U[\rho_h^{k+1}]\Vert_{L^\infty}\le C_T$, we finally get
\[\begin{split}
\int_0^T \int_{\T^d} \partial_t \phi \rho_h =&-\int_0^T \int_{\T^d} F'(\rho_h)\Delta \phi\\
&+\int_0^T \int_{\T^d} \nabla U[\rho_h] \cdot \nabla \phi \rho_h -\int_{\T^d} \phi(0,.) \rho_0 +\delta_h
\end{split}\]
where $\delta_h$ goes to zero as $h\to 0$. Thanks to \pref{estimjko9}, \pref{estimjko10} and \pref{estimjko100} we may pass to the limit on the vanishing sequence $h_n$ to find that the limit $\rho$ is a solution of \pref{nonlineardiffpot}. This completes the proof of Theorem \ref{existporouspot}.

\section{Extension to systems}\label{sec-systems}

Let us now consider the extension of \pref{nonlineardiffpot} to systems for the evolution of $l$ densities $\rho:=(\rho_1,\ldots, \rho_l)$ of  interacting species:
\begin{equation}\label{systemspot}
\partial_t \rho_i= \dive(\rho_i  \nabla (E'_i(\rho_i)+ U_i[\rho]))=0, \; \rho_{i}\vert_{t=0}=\rho_{i,0}, \; i=1,\ldots, l, 
\end{equation}
 on $(0,+\infty)\times \T^d$.  Let us assume that for every  $i=1,\ldots, l$  the map $\rho \in \PT^l\mapsto U_i(\rho)$ fullfills
\begin{equation}\label{hypv0potsyst}
\forall \rho \in \PT^l,   U_i[\rho] \in W^{2, \infty}(\T^d),  \mbox{ and }  U_i[\rho]\ge 0, 
\end{equation}
\begin{equation}\label{hypv1potsyst}
  \sup_{\rho \in \PT^l} \{  \Vert \nabla U_i[\rho] \Vert_{L^\infty} +  \Vert (\Delta  U_i[\rho])_+ \Vert_{L^\infty} \}<+\infty
\end{equation}
and there exists a constant $C$ such that for all $(\rho, \nu)=((\rho_1,\ldots, \rho_l), (\nu_1,\ldots, \nu_l))\in \PT^l\times \PT^l$ and every $i$
\begin{equation}\label{hypv2potsyst}
\Vert \nabla U_i[\rho]-\nabla U_i [\nu]\Vert_{L^{\infty}(\T^d)} \le C\sum_{j=1}^l W_2(\rho_j,\nu_j).
\end{equation}
As in the previous section, we  assume that for each $i$, $E_i$ is a continuous convex function on $\R_+$ with $E_i(0)=0$, $E_i$ is of class $C^2$ on $(0,+\infty)$ and that there are constants  $C>0$ and $m_i\ge 1$ such that
\begin{equation}\label{hypEpotsyst}
E_i''(t)\ge \frac{t^{m_i-2}}{C}, \; tE_i'(t)-E_i(t)\le C(1+t^{m_i}), \; \forall t\in (0,+\infty).
\end{equation}
Finally we assume that the initial condition $\rho_0\in \PT$ satisfies
\begin{equation}\label{hyprho0potsyst}
\sum_{i=1}^l \int_{\T^d} E_i(\rho_{i,0}(x)) \mbox{ d} x<+\infty
\end{equation}
which with \pref{hypEpotsyst} in particular implies that $\rho_{i,0}\in L^{m_i}(\T^d)$. The semi-implicit JKO scheme then takes the following form: given a time step $h>0$, we construct inductively a sequence $\rho_h^k\in \PT^l$ by setting $\rho_h^0=\rho_0$ and, given $\rho_h^k\in \PT^l$ we select $\rho_h^{k+1}$ as a solution of 
\begin{equation}\label{semi-imp-jkosyst}
\inf_{\rho\in \PT^l} \Big\{\frac{1}{2h}  \sum_{i=1}^l  W_2^2 (\rho_i, \rho_{i,h}^k)+ \E(\rho)+ \U(\rho\vert \rho_{h}^k)\Big\}
\end{equation}
where
\[\E(\rho):= \begin{cases} \sum_{i=1}^l \int_{\T^d} E_i(\rho_i(x)) \mbox{ d} x, \; \mbox{ if $E_i(\rho_i)\in L^1$},\\+\infty, \mbox{ otherwise,} \end{cases}\]
and
 \[\U(\rho\vert \nu):=  \sum_{i=1}^l  \int_{\T^d} U_i[\nu] \mbox{d} \rho_i.\]
Extending in a piecewise constant way the sequence $(\rho_h^k)_{k\in \N}$  defines the $\PT^l$-valued curve:
\begin{equation}
\rho_h(t,.):=\rho_h^k \mbox{ for $t\in ((k-1)h, kh]$ and $k\in \N$}.
\end{equation}
Arguing exactly as in the proof detailed in section \ref{sec-proof}, there is strong convergence  in $\Pi_{i=1}^l L^{m_i}((0, T)\times \T^d)$ of a sequence  $\rho_{h_n}$ to some limit  curve $\rho$ and passing to the limit in the Euler-Lagrange for \pref{semi-imp-jkosyst} exactly gives: 

\begin{thm}\label{existporouspotsyst}
Assume \pref{hypv0potsyst}-\pref{hypv1potsyst}-\pref{hypv2potsyst}-\pref{hypEpotsyst}-\pref{hyprho0potsyst}, then \pref{systemspot} admits at least one weak solution. 
\end{thm}

\section{The non potential case} \label{sec-nonpot}

We are now interested in the case where the drift may not be a gradient. More precisely, we consider the following nonlinear diffusion equation:
\begin{equation}\label{nonlineardiff}
\partial_t \rho-\dive(\rho  \nabla (E'(\rho)))+\dive(\rho V[\rho])=0, \; \rho\vert_{t=0}=\rho_0,
\end{equation}
on $(0,T)\times \T^d$.  Denoting by $H^{-1}(\T^d)$ the dual of $H^1(\T^d)$, we assume the following  regularity on the drift term $V[\rho]$:
\begin{equation}\label{hypv1}
\forall \rho \in L^2 \cap \PT,   V[\rho] \in W^{1, \infty}, \; \sup_{\rho\in L^2\cap \PT} \{  \Vert V[\rho] \Vert_{L^\infty} +  \Vert \nabla  V[\rho] \Vert_{L^\infty} \}<+\infty
\end{equation}
and for every $R>0$, there exists a modulus $\omega_R$ such that, for every $(\rho, \nu)\in (L^2(\T^d)\cap \PT)^2$ such that $\Vert \rho\Vert_{H^{-1}(\T^d)}\le R$ and $\Vert \nu \Vert_{H^{-1}(\T^d)}\le R$, one has
\begin{equation}\label{hypv2}
\Vert V[\rho]-V[\nu]\Vert_{L^2(\T^d)} \le \omega_R(\Vert \rho-\nu \Vert_{H^{-1}(\T^d)}).
\end{equation}
Typical examples of velocity fields $\rho\mapsto V[\rho]$ that satisfy the above assumptions are those of the form $V[\rho](x)=\int_{\T^d} B(x,y) \rho(y)dy$ with $B$ smooth enough (but not necessarily a gradient with respect to $x$).

As before, $E$ is convex on $\R_+$ and we define $F'(t):=tE'(t)-E(t)$ so that $F''(t)=tE''(t)$. We make the following assumptions on $F$ (which are satisfied for instance when $E(t)=t^m$ with $m> 1$ or $E(t)=t\log(t)$):
\begin{equation}\label{hypF1}
F\in C^2(\R_+, \R), F(0)=F'(0)=0, \mbox{ $F$ is convex}, 
\end{equation}
\begin{equation}\label{hypF2}
\mbox{$F''$ is nondecreasing, and  for every $\rho>0$, $F''(\rho)>0$}
\end{equation}
and there is a constant $C>0$ such that
\begin{equation}\label{hypF3}
F'(\rho)\le C(1+\rho^2 +F(\rho)), \; \forall \rho\in \R_+.
\end{equation}

As for the initial condition $\rho_0$ we assume that it is a probability density such that
\begin{equation}\label{hyprho0}
\rho_0\in H^1(\T^d), \; F(\rho_0)\in L^1(\T^d), \; F'(\rho_0)\in H^1(\T^d).
\end{equation}

A nonnegative weak solution of the PDE 
\begin{equation}\label{nonlineardiff2}
\partial_t \rho-\Delta(F'(\rho))+\dive(\rho V[\rho])=0, \; \rho\vert_{t=0}=\rho_0.   
\end{equation}
is by definition a function $\rho\in L^2((0,T)\times \T^d, \R_+)$ such that 
\[F'(\rho)\in L^2((0,T), H^1(\T^d))\]
and
\begin{equation}\label{weaksol} 
\int_0^T\int_{\T^d} \left( -\partial_t \phi \rho +\nabla F'(\rho) \cdot \nabla \phi-\rho V[\rho] \cdot \nabla \phi \right) \mbox{d} x \mbox{d} t=\int_{\T^d} \phi(0, x) \rho_0(x) \mbox{d} x
\end{equation}
for every $\phi\in C^1([0,T]\times \T^d)$ such that $\phi(T,.)=0$.

Before we proceed to the existence proof, we need some preliminary results. Let us first study the continuity of the drift term $\rho=\rho(t,x)\mapsto V[\rho(t,.)](x)$. It is easy to see that when  \pref{hypv1} and \pref{hypv2} are satisfied and $\rho^n$ converges strongly in $L^2((0,T)\times \T^d)$ (hence in $L^2((0,T), H^{-1}(\T^d)$) to some $\rho$ then $V[\rho^n]$ converges to $V[\rho]$ in $L^2((0,T)\times \T^d)$, but we wil need a variant in the sequel:

\begin{lem}\label{reguvt}
Assume that \pref{hypv1} and \pref{hypv2} are satisfied. Let $\rho^n$ be a sequence in $L^2((0,T)\times \T^d)$ such that $\partial_t \rho^n \in L^2((0,T), H^{-1}(\T^d))$  
with 
 \begin{equation}\label{boundh-1}
 \sup_n \Vert \partial_t \rho^n \Vert_{L^2((0,T), H^{-1}(\T^d))}<+\infty,
 \end{equation} 
 and $\rho\in L^2((0,T)\times \T^d)$ such that $\rho^n \wto \rho$ in $L^2((0,T)\times \T^d)$, then $V[\rho^n]$ converges to $V[\rho]$ strongly in $L^2((0,T)\times \T^d)$.
\end{lem}

\begin{proof}
First observe that  \pref{boundh-1} implies that for some constant $C$ one has
\begin{equation}\label{holderrho}
\Vert \rho^n(t,.)-\rho^n(s,.)  \Vert_{H^{-1} } \le C \sqrt{\vert t-s\vert}, \; \forall n, \; \forall (t,s)\in (0,T)^2.
\end{equation}
Let $t\in (0,T)$ and for $h\in (0,t)$ define
\[\overline{\rho}^n_{t,h}(x):=\frac{1}{h} \int_{t-h}^{t} \rho^n(s,x) ds, \; \overline{\rho}_{t,h}:=\frac{1}{h} \int_{t-h}^{t} \rho(s,x)ds\]
thanks to \pref{holderrho}, we obtain, for every $n$, $t$ and $h$: 
\begin{equation}
\Vert \rho^n(t,.)-\overline{\rho}^n_{t,h}\Vert_{H^{-1}}\leq C \sqrt{h}, \; \Vert \rho (t,.)-\overline{\rho}_{t,h}\Vert_{H^{-1}}\leq C \sqrt{h}.
\end{equation}
For fixed $h>0$, $\overline{\rho}^n_{t,h}\wto \overline{\rho}_{t,h}$ in $L^2(\T^d)$   as $n\to \infty$, and since the imbedding of $L^2(\T^d)$ into $H^{-1}(\T^d)$ is compact  we also have $\Vert \overline{\rho}^n_{t,h}- \overline{\rho}_{t,h}\Vert_{H^{-1}(\T^d)} \to 0$ as $n\to \infty$. We then get
\[\Vert \rho^n(t,.)-\rho(t,.) \Vert_{H^{-1}} \le 2C \sqrt{h}+  \Vert \overline{\rho}^n_{t,h}- \overline{\rho}_{t,h}\Vert_{H^{-1}(\T^d)}\]
from which we deduce that  $\Vert \rho^n(t,.)-\rho(t,.) \Vert_{H^{-1}}$ tends to $0$. Thanks to \pref{hypv2}, this implies that $\Vert V[\rho^n(t,.)]-V[\rho(t,.)]\Vert_{L^2(\T^d)}$ tends to $0$. The claimed $L^2$ convergence then follows from \pref{hypv1} and Lebesgue's dominated convergence Theorem.

\end{proof}

We now introduce a regularized nonlinearity to approximate \pref{nonlineardiff2} by a uniformly parabolic equation as follows. Let $\eps\in (0, 1)$, let $\delta_\eps$ and $M_\eps$ be respectively the smallest $\rho$ for which $F''(\rho)\ge \eps$ and the largest $\rho$ for which $F''(\rho)\le \eps^{-1}$. Let then $F_\eps$ be defined by
\begin{equation}\label{Feps}
F_\eps(\rho):=\begin{cases} F(\delta_\eps)+ F'(\delta_\eps)(\rho-\delta_\eps)+\frac{\eps}{2} (\rho-\delta_\eps)^2  \mbox{ if } \rho\in [0, \delta_\eps];\\
F(\rho) \mbox{ if } \rho\in[\delta_\eps, M_\eps],\\
F(M_\eps)+ F'(M_\eps)(\rho-M_\eps)+\frac{1}{2\eps} (\rho-M_\eps)^2  \mbox{ if } \rho \ge M_\eps.   
\end{cases}
\end{equation}
Clearly, by construction $F_\eps$ is convex and $C^2$ on $\R_+$ with
\begin{equation}\label{nondeg}
\eps \le F''_\eps\le \frac{1}{\eps} \mbox{ on $\R_+$}
\end{equation}
and $F_\eps$ converges pointwise to $F$ since $\delta_\eps$ and $M_\eps$ converge respectively to $0$ and $+\infty$. In fact, this approximation also has good $\Gamma$-convergence properties:

\begin{lem}\label{lemgamma}
Let $\theta\in L^2((0,T)\times \T^d, \R_+)$, then 
\begin{equation}\label{gamma1}
\lim_{\eps \to 0^+} \int_0^T \int_{\T^d} F_\eps(\theta(t,x))\mbox{d} x \mbox{d} t=\int_0^T \int_{\T^d} F(\theta(t,x))\mbox{d} x \mbox{d} t
\end{equation}
moreover if $\theta_\eps \in L^2((0,T)\times \T^d), \R_+)$ weakly converges to $\theta$ in $\in L^2((0,T)\times \T^d)$, then
 \begin{equation}\label{gamma2}
\liminf_{\eps \to 0^+} \int_0^T \int_{\T^d} F_\eps(\theta_\eps(t,x))\mbox{d} x \mbox{d} t\ge \int_0^T \int_{\T^d} F(\theta(t,x))\mbox{d} x \mbox{d} t
\end{equation}
\end{lem}

\begin{proof}
Fatou's lemma first yields
\[\liminf_{\eps \to 0^+} \int_0^T \int_{\T^d} F_\eps(\theta(t,x))\mbox{d} x \mbox{d} t\ge \int_0^T \int_{\T^d} F(\theta(t,x))\mbox{d} x \mbox{d} t\]
on the other hand
\[ \int_0^T \int_{\T^d} F_\eps(\theta(t,x))\mbox{d} x \mbox{d} t \le  \int_0^T \int_{\T^d} F(\theta(t,x)) \mbox{d} x \mbox{d} t+ \int \int_{\{\theta \le \delta_\eps\}} (F_\eps(\theta)-F(\theta)) \mbox{d} x \mbox{d} t\]
since the second term in the right hand side converges to $0$, we easily deduce \pref{gamma1}. Let us now assume that $\theta_\eps \in L^2((0,T)\times \T^d, \R_+)$ weakly converges to $\theta$ in $\in L^2((0,T)\times \T^d)$. Let $\gamma>0$  (fixed for the moment) and  denote by $F^{\gamma}$ the function defined by
\[F^{\gamma}(\rho)=\begin{cases}  F(\rho) \mbox{ if } \rho \in [0, \gamma], \\
F(\gamma)+F'(\gamma)(\rho-\gamma) \mbox{ if } \rho \ge \gamma
  \end{cases}
\]
by construction $F^\gamma$ is convex and below $F$. For $\eps>0$ small enough so that $\gamma \in [\delta_\eps, M_\eps]$, we similarly define  
\[F^{\gamma}_\eps (\rho)=\begin{cases}  F_\eps(\rho) \mbox{ if } \rho \in [0, \gamma], \\
F(\gamma)+F'(\gamma)(\rho-\gamma) \mbox{ if } \rho \ge \gamma
  \end{cases}
\]
so that $F^\gamma_\eps$ is convex and coincides with $F^\gamma$ on $[\delta_\eps, +\infty)$. We then have
\[\begin{split}
\liminf_{\eps \to 0^+} \int_0^T \int_{\T^d} F_\eps(\theta_\eps(t,x))\mbox{d} x \mbox{d} t\ge \liminf_{\eps \to 0^+} \int_0^T \int_{\T^d} F^\gamma_\eps(\theta_\eps(t,x))\mbox{d} x \mbox{d} t\\
\ge \liminf_{\eps \to 0^+} \int_0^T \int_{\T^d} F^\gamma(\theta_\eps(t,x))\mbox{d} x \mbox{d} t + \liminf_{\eps \to 0^+} \int\int_{\{\theta_\eps \le \delta_\eps\}} (F_\eps(\theta_\eps)-F(\theta_\eps))
\end{split}\]
the second term converges to $0$ whereas by weak lower semi-continuity (thanks to the convexity of $F^\gamma$)  we have
\[\liminf_{\eps \to 0^+} \int_0^T \int_{\T^d} F^\gamma(\theta_\eps(t,x))\mbox{d} x \mbox{d} t \ge \int_0^T \int_{\T^d} F^\gamma(\theta(t,x))\mbox{d} x \mbox{d} t,\]
hence
\[ \liminf_{\eps \to 0^+} \int_0^T \int_{\T^d} F_\eps(\theta_\eps(t,x))\mbox{d} x \mbox{d} t\ge \sup_{\gamma >0} \int_0^T \int_{\T^d} F^\gamma(\theta(t,x))\mbox{d} x \mbox{d} t\]
and then \pref{gamma2} easily follows from the previous inequality, the fact that $F^\gamma$ converges monotonically to $F$ and Beppo-Levi's monotone convergence Theorem. 
\end{proof}

\begin{thm}\label{existporous}
Assume \pref{hypv1}-\pref{hypv2}-\pref{hypF1}-\pref{hypF2}-\pref{hypF3}-\pref{hyprho0}, then \pref{nonlineardiff2} admits at least one weak nonnegative solution. 

\end{thm}

\begin{proof}

The proof proceeds in three steps.

{\bf{Step 1:}} Regularized equation. We first prove existence of a weak solution to the regularized equation: 
\begin{equation}\label{regularized}
\partial_t \rho^\eps-\Delta(F'_\eps(\rho^\eps))+\dive(\rho^\eps V[\rho^\eps])=0, \; \rho^\eps\vert_{t=0}=\rho_0.
\end{equation}
For fixed $\nu\in L^2((0,T)\times \T^d, \R_+)$ the linear parabolic equation:
\begin{equation}\label{defoperator}
\partial_t \rho^\eps-\dive(F''_\eps(\nu) \nabla \rho^\eps)+\dive(\rho^\eps V[\nu])=0, \; \rho^\eps\vert_{t=0}=\rho_0 
\end{equation}
admits a unique weak solution which we denote $\rho^\eps:=T^\eps(\nu)$, moreover $\rho^\eps$ is nonnegative by the maximum principle and  $\rho^\eps \in L^2((0,T), H^1(\T^d))\cap C([0, T], L^2(\T^d))$ and $\partial_t \rho \in L^2((0,T), H^{-1}(\T^d))$ and more precisely, thanks to \pref{hypv1} and \pref{nondeg} there is a constant $C_\eps$ such that:
\begin{equation} \label{bounds}
 \int_0^T \int_{\T^d} (\vert \nabla \rho^\eps\vert^2 + (\rho^\eps)^2)\mbox{d} x \mbox{d} t + \int_0^T \Vert \partial_t \rho^\eps\Vert_{H^{-1}}^2 \le C_\eps.
 \end{equation}
Thanks to \pref{hypv2},  it is easy to check that $T^\eps$ is a continuous map of $L^2((0,T)\times \T^d, \R_+)$. In addition, \pref{bounds} and the Aubin-Lions lemma (see \cite{aubin}, \cite{simon}) imply that $T^\eps(L^2((0,T)\times \T^d, \R_+))$ is relatively compact in $L^2((0,T)\times \T^d)$. Schauder's fixed point Theorem then ensures that $T^\eps$ admits at least one fixed point $\rho^\eps$ i.e. a solution of \pref{regularized}.

\smallskip

{\bf{Step 2:}} A priori estimates. We aim now to derive estimates independent of $\eps$ on $\rho^\eps$. Let $\delta>0$ such that $\delta \in (\delta_\eps, M_\eps)$,  we then take $(\rho^\eps-\delta)_+$ as test-function in \pref{regularized} (which is actually licit since this test-function belongs to $L^2((0,T), H^1(\T^d))$) integrating between $0$ and $t\in [0,T]$ this yields 
\[\ \int_0^t \langle \partial_t \rho^\eps,(\rho^\eps-\delta)_+ \rangle_{H^{-1}, H^1} ds + \int_0^t \int_{\{\rho^\eps \ge \delta\}} F''_\eps(\rho^\eps) \vert \nabla \rho^\eps\vert^2=\int_0^t \int_{\{\rho^\eps \ge \delta\}} \rho^\eps V[\rho^\eps] \cdot   \nabla \rho^\eps\]
hence, using Young's inequality
\[\begin{split}
 \frac{1}{2} \Vert  (\rho^\eps(t, .)-\delta)_+\Vert^2_{L^2} -\frac{1}{2} \Vert  (\rho_0-\delta)_+\Vert^2_{L^2}  + \int_0^t \int_{\{\rho^\eps \ge \delta\}} F''_\eps(\rho^\eps) \vert \nabla \rho^\eps\vert^2\\
 \le C\Big(\frac{\nu}{2}  \int_0^t \int_{\{\rho^\eps \ge \delta\}} \vert \nabla \rho^\eps \vert^2   + \frac{1}{2\nu} \int_0^t \int_{\{\rho^\eps \ge \delta\}}  (\rho^\eps)^2 \Big) \\
\le C\Big(\frac{\nu}{2}  \int_0^t \int_{\{\rho^\eps \ge \delta\}} \vert \nabla \rho^\eps \vert^2   + \frac{1}{\nu} \int_0^t \int_{\{\rho^\eps \ge \delta\}} [(\rho^\eps-\delta)_+^2  +\delta^2]\Big)
\end{split}\]
since $F_\eps''(\delta)=F''(\delta)>0$ and $F''$ nondecreasing, we can choose $\nu$ small enough so that the first term in the right hand side is absorbed by the left hand side of the inequality. Gronwall's lemma then gives 
\begin{equation}\label{L2bound}
\sup_{t\in (0,T)} \Vert \rho^\eps(t,.) \Vert_{L^2} \le C
\end{equation}
for a constant $C$ that does not depend on $\eps$. Next we take $F'_\eps(\rho^\eps)$ as test-function which similarly gives:
\[\begin{split}
\int_{\T^d} F_\eps(\rho^\eps(t,.))-\int_{\T^d} F_\eps(\rho_0) +\int_0^t \int_{\T^d} \vert \nabla F'_\eps(\rho^\eps)\vert^2 =\int_0^t \int_{\T^d} \rho^\eps V[\rho^\eps]\cdot \nabla F'_\eps(\rho^\eps)\\
\le C \Big(  \frac{\nu}{2} \int_0^t \int_{\T^d} \vert \nabla F'_\eps(\rho^\eps)\vert^2 +\frac{1}{2\nu} \int_0^t \int_{\T^d} (\rho^\eps)^2  \Big)
\end{split}\]
using \pref{L2bound} and chosing $\nu$ small enough we thus get
\begin{equation}\label{bounds2}
\sup_{t\in [0,T]} \int_{\T^d}  F_\eps(\rho_\eps(t,.)) +\int_0^T \int_{\T^d} \vert \nabla F'_\eps(\rho^\eps)\vert^2 \le C
\end{equation}
for a constant $C$ not depending on $\eps$. Next we use \pref{hypF3} and \pref{L2bound}-\pref{bounds2} to deduce that
\begin{equation}\label{bounds3}
\sup_{t\in [0,T]} \int_{\T^d}  F'_\eps(\rho^\eps)  \le C
\end{equation}
together with Poincar\'e-Wirtinger inequality, using again \pref{bounds2}, this gives
\begin{equation}\label{bounds4}
\Vert F'_\eps(\rho^\eps)\Vert_{L^2((0,T), H^1(\T^d))} \le C.
\end{equation}

\smallskip

{\bf{Step 3:}} Passing to the limit. Let us set
\begin{equation}
u^\eps:=F'_\eps(\rho^\eps), \; \sigma^\eps:= \nabla u^\eps -\rho^\eps V[\rho^\eps]
\end{equation}
so that \pref{regularized} can be rewritten as
\begin{equation}\label{reglinear}
\partial_t \rho^\eps= \Delta u^\eps -\dive(\rho^\eps V[\rho^\eps])=\dive(\sigma^\eps), \; \rho^\eps\vert_{t=0}=\rho_0.
\end{equation}
We know from the previous step that
\begin{equation}\label{allbounds}
\Vert \rho^\eps\Vert_{L^{\infty}((0,T), L^2(\T^d))} +\Vert \sigma^\eps\Vert_{L^{2}((0,T), L^2(\T^d))}+ \Vert u^\eps \Vert_{L^2( (0,T), H^1(\T^d))} \le C
\end{equation}
as well as
\begin{equation}\label{h-1eps}
\Vert \partial_t \rho^\eps \Vert_{L^2((0,T), H^{-1}(\T^d))}\le C.
\end{equation}

Passing to subsequences if necessary, we may therefore assume that
\begin{equation}\label{wlimits1}
\rho^\eps \wto \rho \mbox{ in $L^2((0,T)\times \T^d)$}, \; u^\eps \wto u  \mbox{ in $L^2((0,T), H^1(\T^d))$}
\end{equation}
and thanks to Lemma \ref{reguvt} and  \pref{h-1eps}, we have
\begin{equation}\label{wlimits2}
\sigma^\eps \wto \sigma:=\nabla u -\rho V[\rho] \mbox{ in $L^2((0,T)\times \T^d)$.}
\end{equation}
Obviously one then has:
\begin{equation}\label{reglinearlimit}
\partial_t \rho=\dive(\sigma)= \Delta u -\dive(\rho V[\rho]), \; \rho\vert_{t=0}=\rho_0.
\end{equation}
So to establish that $\rho$ is a weak nonnegative solution of \pref{nonlineardiff2}, it is enough to prove that $u=F'(\rho)$. Thanks to the convexity of $F$ this amounts to prove that 
\begin{equation}\label{desiredineq}
\int_0^T \int_{\T^d} F(\theta(t,x))\mbox{d} xdt\ge \int_0^T  \int_{\T^d} F(\rho(t,x))\mbox{d} xdt + \int_0^T  \int_{\T^d}  u (\theta-\rho) \mbox{d} x \mbox{d} t
\end{equation}
for every $\theta\in L^2((0,T)\times \T^d, \R_+)$.  By definition of $u^\eps$ we know that
 \begin{equation}\label{desiredineqeps}
\int_0^T \int_{\T^d} F_\eps(\theta(t,x))\mbox{d} xdt\ge \int_0^T  \int_{\T^d} F_\eps(\rho^\eps(t,x))\mbox{d} xdt + \int_0^T  \int_{\T^d}  u^\eps (\theta-\rho^\eps) \mbox{d} x \mbox{d} t.
\end{equation}
Let us  prove that  
\begin{equation}\label{limproduit}
\lim_{\eps} \int_0^T \int_{\T^d} u^\eps \rho^\eps =\int_0^T \int_{\T^d} u \rho. 
\end{equation}
For that purpose, let $\psi^\eps$ be the potential defined by
\begin{equation}
-\Delta \psi^\eps=\rho^\eps, \; \int_{\T^d} \psi^\eps=0, \; \psi^\eps \in H^1(\T^d).
\end{equation}
Thanks to \pref{L2bound}, we have $\psi^\eps \in L^{\infty}((0,T), H^2(\T^d))$ with a bound independendent of $\eps$:
\begin{equation}\label{H2boundpsi}
\Vert \nabla \psi^\eps\Vert_{L^{\infty}((0,T), H^1(\T^d))}\le C.
\end{equation}
As for the time derivative of $\nabla \psi^\eps$ we observe that
\[-\Delta (\partial_t \psi^\eps)=\partial_t \rho^\eps=\dive(\sigma^\eps)\]
so that, thanks to \pref{allbounds}, we have $\partial_t \nabla \psi^\eps\in L^2((0,T)\times \T^d)$ and more precisely 
\[ \Vert \partial_t \nabla  \psi^\eps  \Vert_{ L^2((0,T)\times \T^d)} \le \Vert \sigma^\eps\Vert_{L^2((0,T)\times \T^d)}\le C\]
this proves that $\nabla \psi^\eps$ is bounded in $H^1((0,T)\times \T^d)$, hence converges in $L^2((0,T)\times \T^d)$, up to an extraction if necessary,  to $\psi$ given by
 \begin{equation}
-\Delta \psi=\rho, \; \int_{\T^d} \psi=0, \; \psi \in H^1(\T^d).
\end{equation}
Weak convergence of $\nabla u^\eps$ and strong convergence of $\nabla \psi^\eps$ in $L^2$ then give
\[\begin{split}
\lim_{\eps} \int_0^T \int_{\T^d} u^\eps \rho^\eps= \lim_{\eps} \int_0^T \int_{\T^d} \nabla u^\eps \nabla \psi^\eps\\
= \int_0^T \int_{\T^d} \nabla u \nabla \psi=  \int_0^T \int_{\T^d} u \rho
\end{split}\]
which establishes \pref{limproduit}. Next, we use Lemma \ref{lemgamma}, letting $\eps$ tend to $0^+$, using \pref{limproduit} we obtain inequality \pref{desiredineq} which proves that $u=F'(\rho)$ and so $\rho$ is a weak solution of \pref{nonlineardiff2}, concluding the proof.   

\end{proof}

The previous arguments again clearly adapt to systems. More precisely, let us consider the system for the evolution of $l$ densities $\rho:=(\rho_1,\ldots, \rho_l)$:
\begin{equation}\label{nonlineardiffsystem}
\partial_t \rho_i-\Delta (F'_i(\rho_i))+\dive(\rho_i V_i[\rho])=0, \; \rho_{i}\vert_{t=0}=\rho_{i,0}   
\end{equation}
on $(0,+\infty)\times \T^d$. 
Assuming that each function $F_i$ satisfies \pref{hypF1}-\pref{hypF2}-\pref{hypF3}, that the initial conditions are probability densities which satisfy
\begin{equation}\label{hyprho0i}
\rho_{i,0}\in H^1(\T^d), \; F_i(\rho_{i,0})\in L^1(\T^d), \; F_i'(\rho_{i,0})\in H^1(\T^d), \; \forall i=1,\ldots, l,
\end{equation}
and for every $i=1,\ldots, l$,  the maps $V_i$ satisfy
\begin{equation}\label{hypv1i}
\forall \rho \in L^2( \T^d)^l,   V_i[\rho] \in W^{1, \infty}(\T^d), \; \sup_{\rho\in L^2(\T^d)^l} \{  \Vert V_i[\rho] \Vert_{L^\infty} +  \Vert \nabla  V_i[\rho] \Vert_{L^\infty} \}<+\infty
\end{equation}
and for every $R>0$, there exists a modulus $\omega_R$ such that, for every $(\rho, \nu)\in L^2(\T^d)^l\times L^2(\T^d)^l$ such that $\Vert \rho\Vert_{H^{-1}(\T^d)^l}\le R$ and $\Vert \nu \Vert_{H^{-1}(\T^d)^l}\le R$, one has
\begin{equation}\label{hypv2i}
\Vert V_i[\rho]-V_i[\nu]\Vert_{L^2(\T^d)} \le \omega_R\Big(\sum_{j=1}^l \Vert \rho_j-\nu_j \Vert_{H^{-1}(\T^d)}\Big).
\end{equation}

A direct adaptation of the proof of Theorem \ref{existporous} gives

\begin{thm}\label{existporoussystem}
Assume that each function $F_i$ satisfies \pref{hypF1}-\pref{hypF2}-\pref{hypF3}, and that \pref{hyprho0i}-\pref{hypv1i}-\pref{hypv2i} are satisfied for $i=1,\ldots, l$, then \pref{nonlineardiffsystem} admits at least one weak solution $(\rho_1,\ldots, \rho_l)$ with each $\rho_i$ nonnegative. 
\end{thm}




\section{Uniqueness}\label{sec-uniq}

We end the paper by a general uniqueness argument based on  geodesic convexity. For the purpose of our paper it is enough to consider an internal energy that is a functional defined on the subset $\PTac$ of $\PT$ (consisting of absolutely continuous with respect to the Lebesgue measure on $\T^d$ elements of  $\PT$) and given by 
\[\E(\rho):=\int_{\T^d} E(\rho(x)) \mbox{ d} x , \; \rho \in \PTac,\]
where $E$ : $\R_+\to \R$ is a convex function with $E(0)=0$ and $E$ smooth on $(0,+\infty)$. Given $\rho$ and $\nu$ in $\PTac$,  there is a unique optimal transport map $T$ between $\rho$ and $\nu$ (see Cordero-Erausquin \cite{cordero} and McCann \cite{mcc} for the case of a general Riemannian manifold) i.e. a map such that $T_\#\rho=\nu$ and
\[W_2^2 (\rho, \nu)=\int_{\T^d} d^2(T(x), x) \rho(x) \mbox{d} x.\]
Moreover, this map is given by the gradient of a convex function $T(x)=\nabla \phi(x)$ with $\phi$ convex on $\R^d$ and such that $x\mapsto \nabla \phi(x)-x$ is periodic on $\R^d$ (so that $T$ indeed defines a map from $\T^d$ to itself). In fact, on $\T^d$, we should rather write 
\[T(x)=\exp_x(-\nabla \psi(x)) \mbox{ where } \psi(x)=\frac{1}{2} \vert x\vert^2-\phi(x).\]
Note in particular that $d(x, T(x))= \vert \nabla \psi(x) \vert$ and then
\[W_2(\rho, \nu)=\int_{\T^d} \vert \nabla \psi(x) \vert^2 \rho(x) \mbox{d} x.\]
The Wasserstein geodesic between $\rho$ and $\nu$ is then the curve $t\in [0,1]\mapsto \nu_t$ given by the McCann's interpolation:
\[\nu_t:=T_t \#\rho, \; T_t(x)=\exp_x(-t\nabla \psi(x)),\]
it is indeed a constant speed geodesic:
\[W_2(\nu_t, \nu_s)=\vert t-s\vert W_2(\rho, \nu)\]
and $T_t$ is the optimal transport map from $\rho$ to $\nu_t$. Then the internal energy $\E$ is said to be \emph{displacement convex} whenever for every $\rho$ and $\nu$ in $\PTac$, defining $\nu_t$ as above one has
\[t\in[0,1] \mapsto \E(\nu_t) \mbox{ is convex}.\]
If $E$ satisfies McCann's condition:
\begin{equation}\label{mccond}
r\in (0,+\infty)\mapsto r^{d} E(r^{-d}) \mbox{ is convex nonincreasing}
\end{equation}
then it is well-known that $\E$ is displacement convex (see McCann \cite{mccc}) and that for $\rho$, $\psi$  and $\nu$ as above, one has (setting $F'(\rho)=\rho E'(\rho)-E(\rho)$):
\[\begin{split}
\E(\nu)-\E(\rho)&\ge \lim_{t\to 0^+} \int_{\T^d} \frac{E(\nu_t(x))-E(\rho(x))}{t} \mbox{d} x\\
&= \lim_{t\to 0^+} \int_{\T^d} \frac{1}{t} ( E \Big(\frac{\rho(x)}{\det(DT_t(x))} \Big) \det(DT_t(x))-E(\rho(x))) \mbox{d} x\\
&= - \int_{\T^d} \nabla F'(\rho(x))\cdot \nabla \psi  \mbox{ d} x=
- \int_{\T^d} \nabla E'(\rho(x))\cdot \nabla \psi \rho(x) \mbox{ d} x
\end{split}\]
as soon $\nabla E'(\rho)$   in $L^2(\rho)$. Similarly, letting $S$  ($S(y)=y-\nabla \theta(y)$) be the optimal map from $\nu$ to $\rho$ and using the fact that $S(T(x))=x$ i.e. $\nabla \theta(T(x))=T(x)-x=-\nabla \psi(x)$, we get
\[\begin{split}
\E(\rho)-\E(\nu) &\ge - \int_{\T^d} \nabla E'(\nu(y))\cdot \nabla \theta(y) \nu(y) \mbox{ d} y\\
&=- \int_{\T^d} \nabla E'(\nu(T(x))\cdot \nabla \theta(T(x)) \rho(x) \mbox{ d} x\\
&= \int_{\T^d} \nabla E'(\nu(T(x))\cdot \nabla \psi (x) \rho(x) \mbox{ d} x
\end{split}\]
so that summing the two inequalities above gives
\begin{equation}\label{monot}
0\ge \int_{\T^d} (\nabla E'(\nu(T(x))-\nabla E'(\rho(x)) \cdot \nabla \psi (x) \rho(x) \mbox{ d} x. 
\end{equation}

Now let us consider the system for the evolution of $l$ densities $\rho=(\rho_1, \ldots, \rho_l)$:
\begin{equation}\label{nonlineardiffsystem2}
\partial_t \rho_i=\dive(\rho_i(\nabla E'_i(\rho_i)+ V_i[\rho])), 
\end{equation}
on $(0,+\infty)\times \T^d$. Let us assume that each $E_i$ is a convex function as above, that  $V_i$ maps  $\PT^l$  into $W^{1, \infty}(\T^d)$ and that for some constant $C$ one has
\begin{equation}\label{uniq1}
\vert V_i[\rho](x)- V_i[\rho](y)\vert  \le Cd(x,y), \; \forall \rho\in \PT^l \mbox{ and } (x,y)\in \T^d\times T^d
\end{equation}
and
\begin{equation}\label{uniq2}
\Vert \nabla V_i[\rho]-\nabla V_i [\nu]\Vert_{L^{\infty}(\T^d)} \le C\sum_{j=1}^l W_2(\rho_j,\nu_j), \; \forall (\rho, \nu)\in \PT^l\times \PT^l.
\end{equation}
Then the following uniqueness result holds:

\begin{thm}\label{stab-uniq}
Assume that the drifts $V_i$'s satisfy \pref{uniq1}-\pref{uniq2} and that the $E_i$'s satisfy McCann's condition \pref{mccond}. Let $\rho$ and $\nu$ be two solutions  on $(0,T)\times\T^d$ of   \pref{nonlineardiffsystem2}  such that
\begin{equation}\label{abscont}
\int_{0}^T \sum_{i=1}^l \Vert v_{i,t} \Vert_{L^2(\rho_{i,t})} \mbox{d}t + \int_0^T \sum_{i=1}^l  \Vert w_{i,t}\Vert_{L^2(\nu_{i,t})} \mbox{d}t<+\infty
\end{equation}
with
\begin{equation}\label{defidrifts}
v_{i,t}= \nabla E'_i(\rho_{i,t})+ V_i[\rho_t], \; w_{i,t}= \nabla E'_i(\nu_{i,t})+ V_i[\nu_t]
\end{equation}
then for every $t\in [0,T]$ one has the following stability estimate
\begin{equation}\label{stab-estim}
\sum_{i=1}^l  W_2^2 (\rho_{i,t}, \nu_{i,t})\le e^{4 Ct}  \sum_{i=1}^l W_2^2 (\rho_{i,0}, \nu_{i,0}),
\end{equation}
which in particular gives uniqueness for the Cauchy problem for \pref{nonlineardiffsystem2}.
\end{thm}

\begin{proof}
Using Theorem 5.24 and Corollary 5.25  from \cite{fs}, assumption \pref{abscont} guarantees that $t\mapsto W_2^2(\rho_{i,t}, \nu_{i,t})$ is differentiable for a.e $t\in (0,T)$ with 
\[ \frac{d}{dt} \frac{1}{2} W_2^2 (\rho_{i,t}, \nu_{i,t})=\int_{\T^d} \nabla \psi_{i,t}(x)\cdot (w_{i,t}(T_{i,t}(x))-v_{i,t}(x)) \rho_{i,t}(x) \mbox{d} x  \]
where $\nabla \psi_{i,t}$ is such that $T_{i,t}:=\exp_x(-\nabla \psi_{i,t})$ is the optimal map between $\rho_{i,t}$ and $\nu_{i,t}$. Now thanks to McCann's condition and recalling \pref{monot}, we have
\[ \int_{\T^d} \nabla \psi_{i,t}(x)\cdot ( \nabla E'_i(\nu_{i,t} (T_{i,t}(x)))-\nabla E'_i(\rho_{i,t}) (x))) \rho_{i,t}(x) \mbox{d}x \le 0,\]
so that, using \pref{uniq1}-\pref{uniq2} yields
\[\begin{split}
\frac{d}{dt} \frac{1}{2} W_2^2 (\rho_{i,t}, \nu_{i,t}) &\le \int_{\T^d} \nabla \psi_{i,t}(x)\cdot (V_i[\nu_t](T_{i,t}(x))-V_i[\rho_t](x)) \rho_{i,t}(x) \mbox{d}x \\
&=  \int_{\T^d} \nabla \psi_{i,t}(x)\cdot (V_i[\nu_t](T_{i,t}(x))-V_i[\nu_t](x)) \rho_{i,t}(x) \mbox{d}x\\
&+ \int_{\T^d} \nabla \psi_{i,t}(x)\cdot (V_i[\nu_t](x)-V_i[\rho_t](x)) \rho_{i,t}(x) \mbox{d}x\\
& \le C \int_{\T^d} \vert   \nabla \psi_{i,t}(x)\vert d(x, T_{i,t}(x))  \rho_{i,t}(x) \mbox{d}x\\
&+ C \Big(\int_{\T^d}  \vert   \nabla \psi_{i,t}(x)\vert^2  \rho_{i,t}(x) \mbox{d}x\Big)^{\frac{1}{2}} \sum_{j=1}^l W_2(\rho_{j,t},\nu_{j,t})\\
&=C\left( W_2^2 (\rho_{i,t}, \nu_{i,t})+  W_2 (\rho_{i,t}, \nu_{i,t} )\sum_{j=1}^l W_2(\rho_{j,t},\nu_{j,t})\right).
\end{split}\]
Summing over $i$ and using Gronwall's lemma gives the desired inequality \pref{stab-estim}.

\end{proof}

\begin{rem}
In the uniqueness result stated in Theorem \ref{stab-uniq}, the integrability condition \pref{abscont} is made as an assumption since it ensures absolute continuity for the Wasserstein distance of the curves $\rho$ and $\nu$. In the potential case where $V_i=\nabla U_i$, under the assumptions of Theorem \ref{existporouspotsyst}, it can be checked (see \cite{laborde}) that solutions constructed by the semi-implicit JKO scheme actually satisfy  \pref{abscont}.

\end{rem}

{\textbf{Acknowledgements:}} G.C. gratefully acknowledges the hospitality of
the Mathematics and Statistics Department at UVIC (Victoria, Canada), and the
support from the CNRS, from the ANR, through the project ISOTACE (ANR-12-
MONU-0013) and from INRIA through the \emph{action exploratoire} MOKAPLAN.

\end{document}